\documentclass{amsart}
\usepackage{enumerate,xcolor}
\usepackage{amsmath, amssymb, amsbsy, amsthm}
\usepackage{eucal,url,amssymb,
enumerate,amscd,stmaryrd
}
\usepackage[T1]{fontenc}
\usepackage{lmodern}

\usepackage[pagebackref,colorlinks=true]{hyperref}  
\usepackage[pagebackref]{hyperref}

 \newtheorem{thm}{Theorem}[section]
 \newtheorem{cor}[thm]{Corollary}
 \newtheorem{lem}[thm]{Lemma}
 \newtheorem{prop}[thm]{Proposition}
 \theoremstyle{definition}
 
 \theoremstyle{remark}

 \numberwithin{equation}{section}

\newcommand{\LLL}{\mathcal{L}}
\newcommand{\HH}{\mathcal{H}}
\newcommand{\VV}{\mathcal{V}}

\newcommand{\N}{\mathbb{N}}
\newcommand{\F}{\mathcal{F}}

\newcommand{\LL}{\mathfrak{L}}
\newcommand{\M}{\mathcal{M}}
\newcommand{\X}{\mathfrak{X}}

\newcommand{\n}{\nabla}
\newcommand{\nn}{\widetilde{\n}}
\newcommand{\tD}{\widetilde{\n}^{\parallel}}
\newcommand{\tT}{\widetilde{T}}
\newcommand{\f}{\phi}
\newcommand{\tg}{\widetilde{g}}
\newcommand{\tn}{\widetilde\nabla}
\newcommand{\tR}{\widetilde{R}}
\newcommand{\tF}{\widetilde{F}}
\newcommand{\tS}{\widetilde{S}}
\newcommand{\tQ}{\widetilde{Q}}

\newcommand{\al}{\alpha}

\newcommand{\ta}{\theta}
\newcommand{\om}{\omega}

\newcommand{\D}{\mathrm{d}}
\newcommand{\Span}{\mathrm{span}}

\newcommand{\tr}{\mathrm{tr}}
\newcommand{\Div}{\mathrm{div}}
\newcommand{\Alt}{\mathrm{Alt}}
\newcommand{\sx}{\mathop{\mathfrak{S}}\limits_{x,y,z}}

\newcommand{\thmref}[1]{Theorem~\ref{#1}}

\newcommand{\lemref}[1]{Lemma~\ref{#1}}
\newcommand{\corref}[1]{Corollary~\ref{#1}}

\begin{document}

\title[Pair of associated Schouten-van Kampen connections]
{Pair of associated Schouten-van Kampen connections adapted to an almost paracontact almost paracomplex Riemannian
structure}

\author[H. Manev]{Hristo Manev}

\address{
  	Department of Medical Informatics, Biostatistics and E-Learning\\
	Faculty of Public Health\\
	Medical University of Plovdiv \\
	15A 	Vasil Aprilov Blvd\\
	4002 Plovdiv, Bulgaria
}

\email{hristo.manev@mu-plovdiv.bg}

\author[M. Manev]{Mancho Manev}

\address{
  Department of Algebra and Geometry \\
  Faculty of Mathematics and Informatics \\
  University of Plovdiv Paisii Hilendarski  \\
  24 Tzar Asen St \\
  4000 Plovdiv, Bulgaria\\
	\& \\
	Department of Medical Informatics, Biostatistics and E-Learning\\
	Faculty of Public Health\\
	Medical University of Plovdiv \\
	15A 	Vasil Aprilov Blvd\\
	4002 Plovdiv, Bulgaria
}

\email{mmanev@uni-plovdiv.bg}
\thanks{H.M. and M.M. were partially supported by project MU19-FMI-020 of the Scientific Research Fund, University of Plovdiv Paisii Hilendarski, Bulgaria; H.M. was partially supported by National Scientific Program ”Young Researchers and Post-Doctorants”, Bulgaria.}

\subjclass[2010]{Primary 53C25; Secondary 53C07, 53C50, 53D15.}

\keywords{Distribution, Schouten-van Kampen affine connection, almost paracontact Riemannian manifold, paracontact distribution}


\begin{abstract}
There are introduced and studied a pair of associated Schouten-van Kampen affine connections
adapted to the paracontact distribution and an almost paracontact almost paracomplex Riemannian structure generated by
the pair of associated metrics and their Levi-Civita connections.
By means of the constructed non-symmetric connections,
the basic classes of the manifolds with the considered structure are characterized.
Curvature properties of the studied connections are obtained. A family of examples on a Lie group is constructed.
\end{abstract}

\maketitle

\section{Introduction}

The notion of almost paracontact structure on a differentiable manifold of arbitrary dimension was introduced in \cite{Sato76}. The restriction of this structure on the paracontact distribution is an almost product structure, studied and classified in \cite{Nav83}.
On a manifold equipped with an almost paracontact structure can be considered two kinds of compatible metrics.
If the structure endomorphism induces an isometry on the paracontact distribution of each tangent fibre, then the manifold has an almost paracontact Riemannian structure as in \cite{Sato77,AdatMiya77,SatoMats79}.
In the other case, when the induced transformation is antiisometry, the manifold has a structure of an almost paracontact metric manifold (\cite{NakZam,ZamNak}), where the metric is semi-Riemannian of type $(n+1,n)$.

The object of our considerations are the almost paracontact almost paracomplex Riemannian manifolds. The restriction of the
introduced almost paracontact structure on the paracontact distribution is traceless, i.e. it is an almost paracomplex structure. In \cite{ManSta}, it is given a classification of these manifolds where they are classified under the name of almost paracontact Riemannian manifolds of type $(n,n)$. Their investigation is continued in \cite{ManTav57,ManTav2}.

The Schouten-van Kampen connection preserves by parallelism a pair of complementary dis\-tri\-bu\-tions on a differentiable manifold
endowed with an affine connection \cite{SvK,Ia,BejFarr}. Using the considered connection in \cite{Sol}, there are studied hyperdistributions in Riemannian manifolds.
In \cite{Olsz} and \cite{ManFil}, there are studied the Schouten-van
Kampen connection adapted to an almost (para)contact metric structure and an almost contact B-metric structure, respectively. The studied connection is not natural in general on these manifolds because it preserves the structure tensors except the structure endomorphism.

In the present paper, we introduce and investigate pair of Schou\-ten-van Kampen connections
associated to the pair of Levi-Civita connections and adapted to the paracontact distribution of an almost paracontact almost paracomplex Riemannian manifold.
We characterize the classes of considered manifolds by means of the constructed non-symmetric connections and we obtain some curvature properties.

\section{Almost Paracontact Almost Paracomplex Riemannian Manifolds}

Let us consider an \emph{almost paracontact almost paracomplex Riemannian manifold} denoted by $(\M,\f,\xi,\eta,g)$. This means that $\M$
is a $(2n+1)$-dimensional ($n\in\N$) differentiable manifold equipped with a compatible Rie\-mannian
metric $g$ and an almost
paracontact structure $(\f,\xi,\eta)$, where $\f$ is an endomorphism
of the tangent bundle $T\M$, $\xi$ is a characteristic vector field and $\eta$ is its dual 1-form, such that the following
algebraic relations are satisfied:
\begin{equation}\label{strM}
\begin{array}{c}
\f\xi = 0,\qquad \f^2 = I - \eta \otimes \xi,\qquad
\eta\circ\f=0,\qquad \eta(\xi)=1,\qquad \tr \f=0,\\
g(\f x, \f y) = g(x,y) - \eta(x)\eta(y),\qquad g(x, \xi) = \eta(x),
\end{array}
\end{equation}
denoting the identity transformation on $T\M$ by $I$ (\cite{Sato76}, \cite{ManTav57}).
In the latter equalities and further, $x$, $y$, $z$, $w$ will stand for arbitrary elements of $\X(\M)$, the Lie algebra of tangent vector fields, or vectors in the tangent space $T_p\M$ of $\M$ at an arbitrary
point $p$ in $\M$.

Almost paracontact almost paracomplex Riemannian manifolds, known also as \emph{almost paracontact Riemannian manifolds of type $(n,n)$}, are classified in \cite{ManSta}, where eleven basic classes $\F_1$, $\F_2$, $\dots$, $\F_{11}$ are introduced. This classification is made with respect
to the tensor $F$ of type (0,3) defined by
\begin{equation*}\label{F=nfi}
F(x,y,z)=g\bigl( \left( \nabla_x \f \right)y,z\bigr),
\end{equation*}
where $\n$ is the Levi-Civita connection of $g$.
The following identities are valid:
\begin{equation}\label{F-prop}
\begin{array}{l}
F(x,y,z)=F(x,z,y)=-F(x,\f y,\f z)+\eta(y)F(x,\xi,z)
+\eta(z)F(x,y,\xi),\\
(\n_x\eta)y=g(\n_x\xi,y)=-F(x,\f y, \xi).
\end{array}
\end{equation}

The special class $\F_0$,
determined by the condition $F=0$, is the intersection of the basic classes.

The associated metric $\tg$ of $g$ on $\M$ is defined by
$\tg(x,y)=g(x,\f y)+\eta(x)\eta(y)$. It is shown that $\tg$ is a compatible metric with $(\M,\f,\xi,\eta)$ and it is a pseudo-Riemannian metric of signature $(n + 1, n)$. Therefore,
$(\M,\f,\xi,\eta,\tg)$ is also an almost paracontact almost paracomplex manifold but with a pseudo-Riemannian metric.

The following 1-forms (known also as Lee forms)
are associated with $F$:
\begin{equation*}\label{t}
\theta(z)=g^{ij}F(e_i,e_j,z),\quad
\theta^*(z)=g^{ij}F(e_i,\f e_j,z), \quad \omega(z)=F(\xi,\xi,z),
\end{equation*}
where $\left(g^{ij}\right)$ is the inverse matrix of the
matrix $\left(g_{ij}\right)$ of $g$ with respect to
a basis $\left\{\xi;e_i\right\}$ $(i=1,2,\dots,2n)$ of
$T_p\M$.

Further, we use the following characteristic conditions of the
basic classes: \cite{ManSta}
\begin{equation}\label{Fi}
\begin{array}{rl}
\F_{1}: &F(x,y,z)=\frac{1}{2n}\bigl\{g(\f x,\f y)\ta(\f^2 z) +g(\f x,\f z)\ta(\f^2 y)
\\
&\phantom{F(x,y,z)=\frac{1}{2n}\,\,}
-g(x,\f y)\ta(\f z)-g(x,\f z)\ta(\f y)
\bigr\};\\
\F_{2}: &F(\xi,y,z)=F(x,\xi,z)=0,\quad
              \sx F(x,y,\f z)=0,\quad \ta=0;\\
\F_{3}: &F(\xi,y,z)=F(x,\xi,z)=0,\quad
              \sx F(x,y,z)=0;\\
\F_{4}: &F(x,y,z)=\frac{1}{2n}\ta(\xi)\bigl\{g(\f x,\f y)\eta(z)+g(\f x,\f z)\eta(y)\bigr\};\\
\F_{5}: &F(x,y,z)=\frac{1}{2n}\ta^*(\xi)\bigl\{g( x,\f y)\eta(z)+g(x,\f z)\eta(y)\bigr\};\\
%
\F_{6}: &F(x,y,z)=F(x,y,\xi)\eta(z)+F(x,z,\xi)\eta(y),\quad \\
                &F(x,y,\xi)=F(y,x,\xi)=F(\f x,\f y,\xi),\quad \ta=\ta^*=0; \\
\F_{7}: &F(x,y,z)=F(x,y,\xi)\eta(z)+F(x,z,\xi)\eta(y),\quad \\
         &       F(x,y,\xi)=-F(y,x,\xi)=F(\f x,\f y,\xi); \\
\F_{8}: &F(x,y,z)=F(x,y,\xi)\eta(z)+F(x,z,\xi)\eta(y),\quad \\
         &       F(x,y,\xi)= F(y,x,\xi)=-F(\f x,\f y,\xi); \\
\F_{9}: &F(x,y,z)=F(x,y,\xi)\eta(z)+F(x,z,\xi)\eta(y),\quad \\
         &       F(x,y,\xi)=-F(y,x,\xi)=-F(\f x,\f y,\xi); \\
\F_{10}: &F(x,y,z)=-\eta(x)F(\xi,\f y,\f z); \\
\F_{11}:
&F(x,y,z)=\eta(x)\left\{\eta(y)\om(z)+\eta(z)\om(y)\right\},
\end{array}
\end{equation}
where $\sx$ is the cyclic sum by three arguments $x,y,z$.

The relations between the Lee forms and the divergences  $\Div$ and $\Div^*$ regarding $g$ and $\tg$, respectively, follow directly from  \eqref{F-prop} and they have the form
\begin{equation}\label{divtr}
\ta(\xi)=-\Div^*(\eta),\qquad \ta^*(\xi)=-\Div(\eta).
\end{equation}

As a corollary, the covariant derivative of $\xi$ with respect to $\n$ is determined in each class as follows
\begin{equation}\label{Fi:nxi}
\begin{array}{l}
\F_{1}:\; \n\xi=0;\qquad\qquad
\F_{2}:\; \n\xi=0;\qquad\qquad
\F_{3}:\; \n\xi=0;\\
\F_{4}:\; \n\xi=\frac{1}{2n}\Div^*(\eta)\,\f;\qquad\qquad
\F_{5}:\; \n\xi=\frac{1}{2n}\Div(\eta)\,\f^2;\\
%
\F_{6}:\;  g(\n_x\xi,y)=g(\n_{\f y}\xi,\f x)=g(\n_{\f x}\xi,\f y),\quad \Div(\eta)=\Div^*(\eta)=0; \\
\F_{7}:\; g(\n_x\xi,y)=-g(\n_{\f y}\xi,\f x)=g(\n_{\f x}\xi,\f y); \\
\F_{8}:\; g(\n_x\xi,y)=g(\n_{\f y}\xi,\f x)=-g(\n_{\f x}\xi,\f y); \\
\F_{9}:\; g(\n_x\xi,y)=-g(\n_{\f y}\xi,\f x)=-g(\n_{\f x}\xi,\f y); \\
\F_{10}:\; \n\xi=0; \qquad\qquad
\F_{11}:\; \n\xi=\eta\otimes\f\om^{\sharp},
\end{array}
\end{equation}
where $\sharp$ denotes the musical isomorphism of $T^*\M$ in $T\M$ given by $g$.

Let $\nn$ be the Levi-Civita connection of $\tg$.
Let us denote the potential of $\nn$ regarding $\n$ by $\Phi$, i.e. $\Phi(x,y)=\nn_x y - \n_x y$.

By the well-known Koszul equality for $\tg$ and $\tn$, using \eqref{strM} and \eqref{F-prop},
we obtain the relation between $F$ and $\tF(x,y,z)=\tg(( \nn_x \f )y,z)$
as well as the expression of $\Phi$ in terms of $F$ as follows
\begin{equation}\label{tFF}
\begin{array}{l}
  2\tF(x,y,z)=F(\f y,z,x)-F(y,\f z,x)+F(\f z,y,x)-F(z,\f y,x)\\
  \phantom{2\tF(x,y,z)=}
  +\eta(x)\{F(y,z,\xi)-F(\f z,\f y,\xi)+F(z,y,\xi)-F(\f y,\f z,\xi)\}\\
  \phantom{2\tF(x,y,z)=}
  +\eta(y)\{F(x,z,\xi)-F(\f z,\f x,\xi)+F(x,\f z,\xi)\}\\
  \phantom{2\tF(x,y,z)=}
  +\eta(z)\{F(x,y,\xi)-F(\f y,\f x,\xi)+F(x,\f y,\xi)\},
\end{array}
\end{equation}
\begin{gather}
\begin{array}{l}\label{PhiF}
2\Phi(x,y,z)=F(x,y,\f z)+F(y,x,\f z)-F(\f z,x,y)\\
\phantom{2\Phi(x,y,z)=}
-\eta(x)\{F(y,z,\xi)-F(\f z,\f y,\xi)\}	\\
\phantom{2\Phi(x,y,z)=}
-\eta(y)\{F(x,z,\xi)-F(\f z,\f x,\xi)\}\\
\phantom{2\Phi(x,y,z)=}
-\eta(z)\{F(\xi,x,y)-F(x,y,\xi)+F(x,\f y,\xi)-\omega(\f x)\eta(y)\\
\phantom{2\Phi(x,y,z)=-\eta(z)\{F(\xi,x,y)\,}
-F(y,x,\xi)+F(y,\f x,\xi)-\omega(\f y)\eta(x)\}.
\end{array}
\end{gather}

Obviously, the special class $\F_0$ is determined by any of the following equivalent conditions: $F=0$, $\Phi=0$, $\tF=0$ and $\n=\nn$.

The properties of $\nn_x\xi$ when $(\M,\f,\xi,\eta,\tg)$ belongs to each of the basic classes are determined in a similar way as in \eqref{Fi:nxi}.

\section{Remarkable metric connections regarding the paracontact distribution on the considered manifolds}

Let us consider an almost paracontact almost paracomplex Riemannian manifold $(\M,\f,\xi,\eta,g)$. 
 %
Using the structure $(\xi,\eta)$ on $\M$, there are determined the following two distributions
in the tangent bundle $T\M$ of $\M$
\begin{equation*}\label{HV}
  \HH=\ker(\eta),\qquad \VV=\Span(\xi),
\end{equation*}
called the horizontal distribution and the vertical distribution, respectively.
They are mutually complementary in $T\M$ and orthogonal with respect to $g$ and $\tg$, i.e. $\HH\oplus\VV =T\M$
and $\HH\bot\VV$; moreover, $\HH$ is known also as the paracontact distribution.

We consider the corresponding horizontal and vertical projectors $h:T\M\mapsto\HH$ and $v:T\M\mapsto\VV$.
Since $x=\f^2x+\eta(x)\xi$ for any $x$ in $T\M$, we have
$h(x)=\f^2x$ and $v(x)=\eta(x)\xi$
or equivalently
\begin{equation}\label{Xhv}
  x^h=\f^2x,\qquad x^v=\eta(x)\xi.
\end{equation}

\subsection{The Schouten-van Kampen connections associated to the Levi-Civita connections}

Let us consider the Schouten-van Kampen connections $\n^{\parallel}$ and $\tD$ associated to $\n$ and $\nn$, respectively,
and adapt\-ed to the pair $(\HH, \VV)$. These connections are defined (locally in \cite{SvK}, see also \cite{Ia}) by
\begin{equation}\label{SvK}
\begin{array}{c}
  \n^{\parallel}_x y = (\n_x y^h)^h + (\n_x y^v)^v,\\[6pt]
  \tD_x y = (\nn_x y^h)^h + (\nn_x y^v)^v.
\end{array}
\end{equation}
The latter equalities imply the parallelism of $\HH$ and $\VV$ with
respect to $\n^{\parallel}$ and $\tD$.
Taking into account \eqref{Xhv}, we express the formulae of $\n^{\parallel}$ and $\tD$ in terms of $\n$ and $\nn$, respectively, as follows (cf. \cite{Sol})
\begin{equation}\label{SvK=n}
  \n^{\parallel}_x y = \n_x y -\eta(y)\n_x \xi+(\n_x \eta)\!(y)\,\xi,
\end{equation}
\begin{equation}\label{tSvK=n}
  \tD_x y = \nn_x y -\eta(y)\nn_x \xi+(\nn_x \eta)(y)\xi.
\end{equation}

Obviously, $\n^{\parallel}$ and $\tD$ exist on $(\M,\f,\xi,\eta,g,\tg)$ in each class regarding $F$.

Let us consider the potentials $Q^{\parallel}$ of $\n^{\parallel}$ with respect to $\n$, $\tQ^{\parallel}$ of $\tD$ with respect to $\nn$ and the torsions $T^{\parallel}$ of $\n^{\parallel}$, $\tT$ of $\tD$ defined by $Q^{\parallel}(x,y)=\n^{\parallel}_xy-\n_xy$, $\tQ^{\parallel}(x,y)=\tD_xy-\nn_xy$,  $T^{\parallel}(x,y)= \n^{\parallel}_xy-\n^{\parallel}_yx-[x,y]$ and $\tT(x,y)= \tD_xy-\tD_yx-[x,y]$. Then, they have the following expressions
\begin{gather}\label{Q}
Q^{\parallel}(x,y)=-\eta(y)\n_x \xi+(\n_x \eta)\!(y)\,\xi,
\\
\label{tQ}
\tQ^{\parallel}(x,y)=-\eta(y)\nn_x \xi+(\nn_x \eta)(y)\xi,
\\
\label{T}
T^{\parallel}(x,y)=\eta(x)\n_y \xi-\eta(y)\n_x \xi+\D\eta(x,y)\,\xi,
\\
\label{tT}
\tT^{\parallel}(x,y)=\eta(x)\nn_y \xi-\eta(y)\nn_x \xi+\D\eta(x,y)\xi.
\end{gather}

\begin{thm}\label{thm:D-T}
The Schouten-van Kampen connections $\n^{\parallel}$ and $\tD$ are the unique affine connections having torsions of the form \eqref{T} and \eqref{tT}, respectively, and they preserve the structure $(\xi, \eta,g,\tg)$.
\end{thm}
\begin{proof}
Using \eqref{SvK=n}, we get directly $\n^{\parallel}\xi=\n^{\parallel}\eta=\n^{\parallel}g=0$, i.e. $\xi$, $\eta$ and $g$ are parallel with respect to $\n^{\parallel}$.
Since $\n^{\parallel}$ is a metric connection, it is completely determined by its torsion $T^{\parallel}$.
The spaces of torsions $\{T\}$ and of potentials $\{Q\}$
are isomorphic and the bijection is given by (\cite{Car25})
\begin{gather}
T (x,y,z) = Q(x,y,z) - Q(y,x,z) ,\label{TQ}\\
2Q(x,y,z) = T (x,y,z) - T (y,z,x) + T(z,x,y).\label{QT}
\end{gather}
We verify directly that the potential $Q^{\parallel}$ and the torsion $T^{\parallel}$ of $\n^{\parallel}$, determined by \eqref{Q} and \eqref{T}, respectively, satisfy the latter equalities. This completes the proof for $\n^{\parallel}$. Similarly, we prove for $\tD$.
\end{proof}


\begin{thm}\label{thm:D=n}
The Schouten-van Kampen connection $\n^{\parallel}$ 
coincides with $\n$ if and only if $(\M,\f,\xi,\allowbreak{}\eta,g)$ belongs to the class $\F_1\oplus\F_2\oplus\F_3\oplus\F_{10}$.
\end{thm}
\begin{proof}
According \eqref{SvK=n}, $\n^{\parallel}$ coincides with $\n$ if and only if $\n_x \xi=0$ for any $x$. Having in mind \eqref{Fi:nxi}, this vanishing holds only in the class $\F_1\oplus\F_2\oplus\F_3\oplus\F_{10}$.
\end{proof}
\begin{thm}\label{thm:tD=nn}
The Schouten-van Kampen connection $\tD$ 
coincides with $\nn$ if and only if $(\M,\f,\xi,\allowbreak{}\eta,\tg)$ belongs to the class $\F_1\oplus\F_2\oplus\F_3\oplus\F_{9}$.
\end{thm}
\begin{proof}
The connection $\tD$ coincides with $\nn$ if and only if $\nn \xi$ vanishes. This condition holds if and only if $\widetilde{F}$ satisfies the conditions of $F$ in \eqref{Fi} for $\F_1\oplus\F_2\oplus\F_3\oplus\F_{9}$.
\end{proof}

Taking into account \eqref{tFF}, we prove immediately the following
\begin{lem}\label{lem:U1}
The manifold $(\M,\f,\xi,\eta,g)$ belongs to the class $\F_1\oplus\F_2\oplus\F_3\oplus\F_{10}$ if and only if the manifold  $(\M,\f,\xi,\eta,\tg)$ belongs to the class $\F_1\oplus\F_2\oplus\F_3\oplus\F_{9}$.
\end{lem}

Then, \thmref{thm:D=n}, \thmref{thm:tD=nn} and \lemref{lem:U1} imply the following
\begin{thm}\label{thm:D=ntD=nn}
Let $\n^{\parallel}$ and $\tD$ be the Schouten-van Kampen connections associated to $\n$ and $\nn$, respectively, and adapted to the pair $(\HH,\VV)$ on $(\M,\f,\xi,\eta,g,\tg)$. Then the following assertions are equivalent:
\begin{enumerate}
  \item $\n^{\parallel}$ coincides with $\n$;
  \item $\tD$ coincides with $\nn$;
  \item $(\M,\f,\xi,\eta,g)$ belongs to $\F_1\oplus\F_2\oplus\F_3\oplus\F_{10}$;
  \item $(\M,\f,\xi,\eta,\tg)$ belongs to $\F_1\oplus\F_2\oplus\F_3\oplus\F_{9}$.
\end{enumerate}
\end{thm}
\begin{cor}\label{cor:D=ntD=nn}
Let $\n^{\parallel}$ and $\tD$ be the Schouten-van Kampen connections associated to $\n$ and $\nn$, respectively, and adapted to the pair $(\HH,\VV)$ on $(\M,\f,\xi,\eta,g,\tg)$. If $\tD\equiv\n$ or $\n^{\parallel}\equiv\nn$ then the four connections $\n^{\parallel}$, $\tD$, $\n$ and $\nn$ coincide. The latter coincidences are equivalent to the condition $(\M,\f,\xi,\eta,g)$ and $(\M,\f,\xi,\eta,\tg)$ belong to $\F_0$.
\end{cor}

We obtain the following relation between $\n^{\parallel}$ and $\tD$, using \eqref{tSvK=n} and $\Phi$,
\begin{equation}\label{tD=D}
  \tD_x y = \n^{\parallel}_x y + \Phi(x,y) -\eta(\Phi(x,y))\xi -\eta(y)\Phi(x,\xi).
\end{equation}

The two connections $\tD$ and $\n^{\parallel}$ coincide if and only if $\Phi(x,y)=\eta(\Phi(x,y))\xi +\eta(y)\Phi(x,\xi)$ which is equivalent to $\Phi(x,y)=\eta(\Phi(x,y))\xi +\eta(x)\eta(y)\Phi(\xi,\xi)$ because $\Phi$ is symmetric.
By virtue of \eqref{PhiF} and the latter expression of $\Phi$, we obtain
\begin{equation}\label{F_D=0}
  F(x,y,z)=F(x,y,\xi)\eta(z)+F(x,z,\xi)\eta(y),
\end{equation}
which determines the class $\F_4\oplus\cdots\oplus\F_9\oplus\F_{11}$.
Then, the following assertion is valid.
\begin{thm}\label{thm:tD=D}
The Schouten-van Kampen connections $\tD$ and $\n^{\parallel}$ associated to $\nn$ and $\n$, respectively, and adapt\-ed to the pair  $(\HH,\VV)$
coincide with each other if and only if the manifold belongs to the class $\F_4\oplus\cdots\oplus\F_9\oplus\F_{11}$.
\end{thm}

\subsection{The conditions for natural connections $\n^{\parallel}$ and $\tD$ for $(\f,\xi,\eta,g,\tg)$}

It is known that a connection is called natural for a structure $(\f,\xi,\eta,g,\tg)$ when all of the structure tensors are parallel with respect to this connection. According to \thmref{thm:D-T}, $\n^{\parallel}$ preserves $(\xi,\eta,g)$. However, $\n^{\parallel}$ is not a natural connection for the studied structures, because $\n^{\parallel}\f$ (therefore $\n^{\parallel}\tg$, too) is generally not zero.


\begin{thm}\label{thm:D-nat}
The Schouten-van Kampen connection $\n^{\parallel}$ 
is a natural connection for the structure $(\f,\xi,\eta,g)$ if and only if $(\M,\f,\xi,\eta,g)$ belongs to the class $\F_4\oplus\cdots\oplus\F_9\oplus\F_{11}$.
\end{thm}
\begin{proof} Bearing in mind \eqref{SvK=n}, we obtain the covariant derivative of $\f$ with respect to $\n^{\parallel}$ as follows
\begin{equation}\label{Df}
  (\n^{\parallel}_x\f)y=(\n_x\f)y+\eta(y)\f\n_x\xi-\eta(\n_x\f y)\xi.
\end{equation}
Then, $\n^{\parallel}\f$ vanishes if and only if $(\n_x\f)y=-\eta(y)\f\n_x\xi+\eta(\n_x\f y)\xi$ holds, which is equivalent to \eqref{F_D=0}.
Bearing in mind the proof of \thmref{thm:tD=D}, we find that the class with natural connection $\n^{\parallel}$ is the class in the statement.
\end{proof}

Taking into account \thmref{thm:D=n} and \thmref{thm:D-nat}, we obtain the following
\begin{cor}
The class of all almost paracontact almost paracomplex Riemannian  manifolds can be decomposed orthogonally to the subclass of the manifolds with coinciding connections $\n^{\parallel}$ and $\n$ and the subclass of manifolds with natural $\n^{\parallel}$.
\end{cor}

Taking into account  \eqref{tD=D}, we get the following relation between the covariant derivatives of $\f$ with respect to $\tD$ and $\n^{\parallel}$
\begin{equation}\label{tDfiDfi}
(\tD_x\f)y=(\n^{\parallel}_x\f)y+\Phi(x,\f y)-\f\Phi(x,y)+\eta(y)\f\Phi(x,\xi)-\eta(\Phi(x,\f y))\xi.
\end{equation}

Therefore, we establish that $\tD\f$ and $\n^{\parallel}\f$ coincide if and only if the condition
\[
\Phi(x,\f^2 y,\f^2 z)=-\Phi(x,\f y,\f z)
 \]
 holds, which is fulfilled only when $(\M,\f,\xi,\eta,g)$ is in the class $\F_3\oplus\F_4\oplus\F_5\oplus\F_6\oplus\F_7\oplus\F_{11}$.
Similarly, we establish that $(\M,\f,\xi,\eta,\tg)$ belongs to the same class and therefore we proved the following
\begin{thm}\label{thm:tDfi=Dfi}
The covariant derivatives of $\f$ with respect to the Schouten-van Kampen connections $\n^{\parallel}$ and $\tD$ %
coincide if and only if both of the manifolds $(\M,\f,\xi,\allowbreak{}\eta,\allowbreak{}g)$ and $(\M,\f,\xi,\eta,\tg)$  belong to the class $\F_3\oplus\F_4\oplus\F_5\oplus\F_6\oplus\F_7\oplus\F_{11}$.
\end{thm}

Bearing in mind \eqref{PhiF}, \eqref{Df} and \eqref{tDfiDfi}, we obtain that $\tD\f=0$ is equivalent to
\[
F(\f y,\f z,x)+F(\f^2 y,\f^2 z,x)-F(\f z,\f y,x)-F(\f^2 z,\f^2 y,x)=0.
\]
Then, the latter equality and \eqref{Fi} imply the truthfulness of the following

\begin{thm}\label{thm:tD-nat}
The Schouten-van Kampen connection $\tD$ 
is a natural connection for the structure $(\f,\xi,\eta,\tg)$ if and only if $(\M,\f,\xi,\eta,\tg)$ belongs to the class $\F_1\oplus\F_2\oplus\F_4\oplus\F_5\oplus\F_6\oplus\F_7\oplus\F_{11}$.
\end{thm}

Combining \thmref{thm:D-nat}, \thmref{thm:tDfi=Dfi} and \thmref{thm:tD-nat}, we get the validity of the following

\begin{thm}\label{thm:DtD-nat}
The Schouten-van Kampen connections $\n^{\parallel}$ and $\tD$ 
are natural connections on $(\M,\f,\xi,\eta,g,\tg)$ if and only if
$(\M,\f,\xi,\eta,g)$ and $(\M,\f,\xi,\eta,\tg)$ belong to the class $\F_4\oplus\F_5\oplus\F_6\oplus\F_7\oplus\F_{11}$.
\end{thm}

\section{Torsion properties 
of the pair of connections $\n^{\parallel}$ and $\tD$}

Since $g(\xi,\xi)=1$ implies $g(\n_x\xi,\xi)=0$, it follows that $\n_x\xi\in\HH$.
The shape operator $S:\HH\mapsto\HH$ for $g$ is defined as usually by $S(x)=-\n_x\xi$.

Then, using the relations between $T^{\parallel}$, $Q^{\parallel}$ and $S$ given in \eqref{Q}, \eqref{T}, \eqref{TQ}, \eqref{QT}, we have that the properties of the torsion, the potential and the shape operator for $\n^{\parallel}$ are related.

Horizontal and vertical components of $Q^{\parallel}$ and $T^{\parallel}$ given in \eqref{Q} and \eqref{T}, respectively, are the following
\begin{equation}\label{QTShv}
\begin{array}{ll}
Q^{\parallel h}=S\otimes\eta,\qquad & Q^{\parallel v}=-S^{\flat}\otimes\xi,
\\
T^{\parallel h}=-\eta\wedge S,\qquad & T^{\parallel v}=-2\Alt(S^{\flat})\otimes\xi,
\end{array}
\end{equation}
where $S^{\flat}(x,y)=g(S(x),y)$, i.e.  $S^{\flat}=-\n\eta$, whereas $\wedge$ and $\Alt$ denote the exterior product and the alternation, respectively.

Using the vertical components of $Q^{\parallel}$ and $T^{\parallel}$ from \eqref{QTShv}, we obtain immediately
\begin{thm}\label{thm:equiv1}
The following properties are equivalent:
\begin{enumerate}
  \item
        $\n\eta$ is symmetric
  \item
				$\eta$ is closed, i.e. $\D\eta=0$
  \item
			  $Q^{\parallel v}$ is symmetric
  \item
				$T^{\parallel v}$ vanishes
  \item
				$S$ is self-adjoint regarding $g$
  \item
				$S^{\flat}$ is symmetric
  \item
				$\M\in\F_1\oplus\F_2\oplus\F_3\oplus\F_4\oplus\F_5\oplus\F_6\oplus\F_9\oplus\F_{10}$.
\end{enumerate}
\end{thm}

\begin{thm}\label{thm:equiv2}
The following properties are equivalent:
\begin{enumerate}
  \item
      $\n\eta$ is skew-symmetric
   \item
					$\xi$ is Killing with respect to $g$, i.e. $\LL_{\xi}g=0$
   \item
			   $Q^{\parallel v}$ is skew-symmetric
   \item
			$S$ is anti-self-adjoint regarding $g$
   \item
			$S^{\flat}$ is skew-symmetric
   \item
			$\M\in\F_1\oplus\F_2\oplus\F_3\oplus\F_7\oplus\F_8\oplus\F_{10}$.
\end{enumerate}
\end{thm}

\begin{thm}\label{thm:equiv3}
The following properties are equivalent:
\begin{enumerate}
  \item
        $\n\eta=0$
  \item
        $\D\eta=\LL_{\xi}g=0$
  \item
        $\n\xi=0$
  \item
        $S=0$
  \item
        $S^{\flat}=0$
  \item
        $\n^{\parallel}=\n$
  \item
        $\M\in\F_1\oplus\F_2\oplus\F_3\oplus\F_{10}$.
\end{enumerate}
\end{thm}

In the same manner, we obtain similar linear relations between the torsion, the potential and the shape operator for $\tD$.

The equality $\tg(\xi,\xi)=1$ implies $\tg(\nn_x\xi,\xi)=0$ and therefore $\nn\xi\in\HH$ holds true.
The equality $\tS(x)=-\nn_x\xi$ defines the shape operator $\tS:\HH\mapsto\HH$ for $\tg$.

Now, having in mind \eqref{tQ} and \eqref{tT}, we express the horizontal and vertical components of $\tQ^{\parallel}$ and $\tT$ of $\tD$ as follows
\begin{equation}\label{tQThv}
\begin{array}{ll}
\tQ^{\parallel h}=\tS\otimes\eta,\qquad & \tQ^{\parallel v}=-\tS^{\flat}\otimes\xi,
\\
\tT^{\parallel h}=-\eta\wedge\tS,\qquad & \tT^{\parallel v}=-2\Alt(\tS^{\flat})\otimes\xi,
\end{array}
\end{equation}
where we denote $\tS^{\flat}(x,y)=\tg(\tS(x),y)$.

Bearing in mind that $(\nn_x \eta)(y)=(\n_x \eta)(y)-\eta(\Phi(x,y))$ and $\nn_x \xi=\n_x \xi +\Phi(x,\xi)$, we get
\begin{equation}\label{tSSPhi}
\tS(x)=S(x)-\Phi(x,\xi),\qquad \tS^{\flat}(x,y)=S^{\flat}(x,\f y)-\Phi(\xi,x,\f y).
\end{equation}
Moreover,
\eqref{Q}, \eqref{T}, \eqref{QTShv}, \eqref{tQ}, \eqref{tT} and \eqref{tQThv} imply the following relations
\begin{equation*}\label{tQQ}
\begin{array}{ll}
\tQ^{\parallel h}=Q^{\parallel h}-(\xi\lrcorner\Phi)\otimes\eta,\qquad & \tQ^{\parallel v}=Q^{\parallel v}-(\eta\circ\Phi)\otimes\xi,
\\
\tT^{\parallel h}=T^{\parallel h}+\eta\wedge(\xi\lrcorner\Phi),\qquad & \tT^{\parallel v}=T^{\parallel v},
\end{array}
\end{equation*}
where $\lrcorner$ denotes the interior product.

Subsequently, using the latter equalities and \eqref{tSSPhi}, we obtain the following formulae
\begin{gather*}\label{tQQS}
\tQ^{\parallel}=Q^{\parallel}+(\tS-S)\otimes\eta-(\tS^{\flat}-S^{\flat})\otimes\xi,
\\
\label{tTTS}
\tT^{\parallel}=T^{\parallel}+(\tS-S)\wedge\eta;
\\
\begin{array}{ll}\label{tQQTThvS}
\tQ^{\parallel h}=Q^{\parallel h}+(\tS-S)\otimes\eta,\qquad & \tQ^{\parallel v}=Q^{\parallel v}-(\tS^{\flat}-S^{\flat})\otimes\xi,
\\
\tT^{\parallel h}=T^{\parallel h}+(\tS-S)\wedge\eta,\qquad & \tT^{\parallel v}=T^{\parallel v}.
\end{array}
\end{gather*}

Bearing in mind the obtained results, we get the truthfulness of the following

\begin{thm}\label{thm:equiv-t1}
The following properties are equivalent:
\begin{enumerate}
  \item
        $\nn\eta$ is symmetric
  \item
      $\eta$ is closed
  \item
        $\tQ^{\parallel v}$ is symmetric
  \item
        $\tT^v$ vanishes
  \item
        $\tS$ is self-adjoint regarding $\tg$
  \item
        $\tS^{\flat}$ is symmetric
  \item
        $(\M,\f,\xi,\eta,\tg)\in\F_1\oplus\F_2\oplus\F_3\oplus\F_4\oplus\F_5\oplus\F_6\oplus\F_{9}\oplus\F_{10}$.
\end{enumerate}
\end{thm}

\begin{thm}\label{thm:equiv-t2}
The following properties are equivalent:
\begin{enumerate}
  \item
        $\nn\eta$ is skew-symmetric
  \item
        $\xi$ is Killing with respect to $\tg$, i.e. $\LL_{\xi}\tg=0$
  \item
        $\tQ^{\parallel v}$ is skew-symmetric
  \item
        $\tS$ is anti-self-adjoint regarding $\tg$
  \item
        $\tS^{\flat}$ is skew-symmetric
  \item
        $(\M,\f,\xi,\eta,\tg)\in\F_1\oplus\F_2\oplus\F_3\oplus\F_7\oplus\F_{9}$.
\end{enumerate}
\end{thm}

\begin{thm}\label{thm:equiv-t3}
The following properties are equivalent:
\begin{enumerate}
  \item
        $\nn\eta=0$
  \item
        $\D\eta=\LL_{\xi}\tg=0$
  \item
        $\nn\xi=0$
  \item
        $\tS=0$
  \item
        $\tS^{\flat}=0$
  \item
        $\tD=\nn$
  \item
        $(\M,\f,\xi,\eta,\tg)\in\F_1\oplus\F_2\oplus\F_3\oplus\F_{9}$.
\end{enumerate}
\end{thm}

\section{Curvature properties of the pair of connections $\n^{\parallel}$ and $\tD$}

Let $R$ be the curvature tensor of $\n$, i.e. $R=[\n\ , \n\ ]
- \n_{[\ ,\ ]}$ and let the corresponding $(0,4)$-tensor be
determined by $R(x,y,z,w)=g(R(x,y)z,w)$. The Ricci tensor
$\rho$ and the scalar curvature $\tau$  are defined by
$\rho(y,z)=g^{ij}R(e_i,y,z,e_j)$ and
$\tau=g^{ij}\rho(e_i,e_j)$.

An arbitrary non-degenerate 2-plane $\al$ in
$T_p\M$ with an arbitrary basis $\{x,y\}$ has the following sectional curvature
$
k(\al;p)=\frac{R(x,y,y,x)}{\pi_1(x,y,y,x)}
$ with respect to $g$ and $R$, where $\pi_1(x,y,z,w)=g(y,z)g(x,w)-g(x,z)g(y,w)$.
A 2-plane $\al$ is said to be a \emph{$\xi$-section}, a \emph{$\f$-holomorphic section} or a \emph{$\f$-totally real section}
if $\xi \in \al$, $\al= \f\al$ or $\al\bot \f\al$ regarding $g$, respectively. The latter type of sections exists for $\dim \M\geq 5$.

Let $R^{\parallel}$, $\rho^{\parallel}$, $\tau^{\parallel}$ and $k^{\parallel}$ denote the curvature tensor, the Ricci tensor, the scalar curvature and the sectional curvature of $\n^{\parallel}$, respectively. The corresponding $(0,4)$-tensor is
determined by $R^{\parallel}(x,y,z,w)=g(R^{\parallel}(x,y)z,w)$. Analogously, by
$\tR$, $\widetilde\rho$, $\widetilde\tau$, $\widetilde{k}$ and
$\widetilde{R}^\parallel$, $\widetilde{\rho}^\parallel$, $\widetilde{\tau}^\parallel$, $\widetilde{k}^\parallel$
we denote the corresponding quantities for the connections $\tn$ and $\tD$, respectively, where the $(0,4)$-tensors of $\tR$ and $\widetilde{R}^\parallel$ are obtained by $\tg$.

\begin{thm}\label{thm:KR}
The curvature tensors of 
$\n^{\parallel}$ and 
$\n$ ($\tD$ and $\tn$, respectively) are related as follows
\begin{equation}\label{RDR}
\begin{array}{l}
  R^{\parallel}(x,y,z,w)=R\left(x,y,\f^2z,\f^2w\right)+\pi_1\bigl(S(x),S(y),z,w\bigr),\\
  \widetilde{R}^\parallel(x,y,z,w)=\tR\left(x,y,\f^2z,\f^2w\right)+\widetilde\pi_1\bigl(\tS(x),\tS(y),z,w\bigr),
\end{array}
\end{equation}
where $\widetilde\pi_1(x,y,z,w)=\tg(y,z)\tg(x,w)-\tg(x,z)\tg(y,w)$.
\end{thm}
\begin{proof}
Using \eqref{SvK=n} and $g(\n_x\xi,\xi)=0$ for arbitrary $x$ and $\n^{\parallel}\xi=0$, we get
\[
\begin{array}{l}
  R^{\parallel}(x,y)z=R(x,y)z-\eta(z)R(x,y)\xi-\eta(R(x,y)z)\xi\\
  \phantom{R^{\parallel}(x,y)z=}
  -g\left(\n_x\xi,z\right)\n_y\xi+g\left(\n_y\xi,z\right)\n_x\xi,
\end{array}
\]
which implies the first relation in \eqref{RDR}.

The second equality in \eqref{RDR} follows in similar way, but in terms of $\tD$, $\tn$ and their corresponding metric $\tg$.
\end{proof}

Let $\widetilde\tr$ denotes the trace with respect to $\tg$. We get the following
\begin{cor}\label{cor:Ric}
The Ricci tensors of 
$\n^{\parallel}$ and
$\n$ ($\tD$ and $\tn$, respectively) are related as follows
\begin{equation}\label{RicDRic}
\begin{array}{l}
  \rho^{\parallel}(y,z)=\rho(y,z)-\eta(z)\rho(y,\xi)-R(\xi,y,z,\xi)\\
  \phantom{\rho^{\parallel}(y,z)=\rho(y,z)}
-g(S^2(y),z)+\tr(S)g(S(y),z),\\
  \widetilde{\rho}^\parallel(y,z)=\widetilde\rho(y,z)-\eta(z)\widetilde\rho(y,\xi)-\tR(\xi,y,z,\xi)\\
\phantom{  \widetilde{\rho}^\parallel(y,z)=\widetilde\rho(y,z)}
-\tg(\tS^2(y),z)+\widetilde\tr(\tS)\tg(\tS(y),z).
\end{array}
\end{equation}
\end{cor}

Using \eqref{PhiF} and \eqref{F-prop}, we get that $g^{ij}\Phi(\xi,e_i,e_j)=0$. Therefore, bearing in mind \eqref{divtr} and the definitions of $S$ and $\tS$, we have that
$\tr(S)=\widetilde\tr(\tS)=-\Div(\eta)$.

The definition of the shape operator implies
$
R(x,y)\xi =-\left(\n_x S\right)y+\left(\n_y S\right)x.
$
Then, the latter formula and $S(\xi)=-\n_{\xi} \xi=-\f\omega^{\sharp}$ lead to the following expression
\[
R(\xi,y,z,\xi) = g\bigl(\left(\n_{\xi} S\right)y-\left(\n_y S\right)\xi,z\bigr)
=g\bigl(\left(\n_{\xi} S\right)y-\n_y S(\xi)- S(S(y)),z\bigr).
\]
Then, taking the trace of the latter equalities and using
the relations
\[
\Div(\omega\circ\f)=g^{ij}\left(\n_{e_i} \omega\circ\f\right)e_j=g^{ij}g\left(\n_{e_i}\f \omega^{\sharp},e_j\right)=-\Div(S(\xi)).
\]
we obtain
\begin{equation}\label{ricxixi}
\rho(\xi,\xi)= \tr(\n_{\xi}S)-\Div(S(\xi))-\tr(S^2).
\end{equation}

Similarly, we get the equalities for the quantities with respect to $\tg$, i.e.
\begin{equation}\label{tricxixi}
\widetilde\rho(\xi,\xi)= \widetilde\tr(\nn_{\xi}\tS)-\widetilde\Div(\tS(\xi))-\widetilde\tr(\tS^2).
\end{equation}
Bearing in mind the latter results and \corref{cor:Ric}, we obtain the following

\begin{cor}\label{cor:tau}
The scalar curvatures of 
$\n^{\parallel}$ and 
$\n$ ($\tD$ and $\tn$, respectively) are related as follows
\begin{equation*}\label{tauDtau}
\begin{array}{l}
  \tau^{\parallel}=\tau-2\rho(\xi,\xi)-\tr(S^2) + (\tr(S))^2,\\
  \widetilde\tau^{\parallel}=\widetilde\tau-2\widetilde\rho(\xi,\xi)-\widetilde\tr(\tS^2) + (\widetilde\tr(\tS))^2,
\end{array}
\end{equation*}
where $\rho(\xi,\xi)$ and $\widetilde\rho(\xi,\xi)$ are expressed in \eqref{ricxixi} and \eqref{tricxixi}, respectively.
\end{cor}

Moreover, using  \thmref{thm:KR} we get the following

\begin{cor}\label{cor:kDk}
The sectional curvatures of an arbitrary 2-plane $\alpha$ at $p\in \M$, for an arbitrary basis $\{x,y\}$ of $\alpha$, regarding 
$\n^{\parallel}$ and 
$\n$ ($\tD$ and $\tn$, respectively) are related as follows
\begin{equation}\label{kDk}
\begin{split}
  k^{\parallel}(\al;p)&=k(\al;p)\\
  &\phantom{=}+\frac{\pi_1(S(x),S(y),y,x)-\eta(x)R(x,y,y,\xi)-\eta(y)R(x,y,\xi,x)}{\pi_1(x,y,y,x)},\\
  \widetilde k^{\parallel}(\al;p)&=\widetilde k(\al;p)\\
  &\phantom{=}+\frac{\widetilde\pi_1(\tS(x),\tS(y),y,x)-\eta(x)\tR(x,y,y,\xi)
  -\eta(y)\tR(x,y,\xi,x)}{\widetilde\pi_1(x,y,y,x)}.
\end{split}
\end{equation}
\end{cor}

Let $\al_{\xi}$ be a $\xi$-section  at $p\in \M$ with a basis $\{x,\xi\}$. Then, using \eqref{kDk}, $g(S(x),\xi)=0$ and $\tg(\tS(x),\xi)=0$ for arbitrary $x$, we get that
\[
k^{\parallel}(\al_{\xi};p)=0, \quad \widetilde{k}^\parallel(\al_{\xi};p)=0.
\]
Let $\al_{\f}$ be a $\f$-section  at $p\in \M$ with a basis $\{x,y\}$. Then, using \eqref{kDk} and $\eta(x)=\eta(y)=0$, we obtain that
\[
k^{\parallel}(\al_{\f};p)=k(\al_{\f};p)
+\frac{\pi_1(S(x),S(y),y,x)}{\pi_1(x,y,y,x)},
\]
\[
\widetilde k^{\parallel}(\al_{\f};p)=\widetilde k(\al_{\f};p)
+\frac{\widetilde\pi_1(\tS(x),\tS(y),y,x)}{\widetilde\pi_1(x,y,y,x)}.
\]
Let $\al_{\bot}$ be a $\f$-totally real section orthogonal to $\xi$ with a basis $\{x,y\}$. Then, using \eqref{kDk} and $\eta(x)=\eta(y)=0$, we get that
\[
k^{\parallel}(\al_{\bot};p)=k(\al_{\bot};p)
+\frac{\pi_1(S(x),S(y),y,x)}{\pi_1(x,y,y,x)},
\]
\[
\widetilde k^{\parallel}(\al_{\bot};p)=\widetilde k(\al_{\bot};p)
+\frac{\widetilde \pi_1(\tS(x),\tS(y),y,x)}{\widetilde \pi_1(x,y,y,x)}.
\]

Let us remark that, in the case when $\al$ is a $\f$-totally real section non-orthogonal to $\xi$ regarding $g$ or $\tg$, the relation between the corresponding sectional curvatures regarding $\n^{\parallel}$ and $\n$ ($\tD$ and $\nn$, respectively) is just the first (respectively, the second) equality in \eqref{kDk}.


\section{A family of Lie groups as manifolds of the studied type}\label{sec-exm}

Let us consider as an example
the $(2n+1)$-dimensional almost paracontact almost paracomplex
Riemannian
manifold $(\LLL,\f, \xi, \eta, g)$, which is introduced in \cite{ManTav57}.
That means we have
a real connected Lie group $\LLL$ and its
associated Lie algebra with a global basis $\{E_{0},E_{1},\dots,
E_{2n}\}$ of left invariant vector fields on $\LLL$ defined by:
\begin{equation}\label{com}
    [E_0,E_i]=-a_iE_i-a_{n+i}E_{n+i},\qquad
    [E_0,E_{n+i}]=-a_{n+i}E_i+a_{i}E_{n+i},
\end{equation}
where $a_1,\dots,a_{2n}$ are real constants and $[E_j,E_k]$ vanishes in
other cases.
The structure $(\f,\xi,\eta)$   is determined
for any
${i\in\{1,\dots,n\}}$ as follows:
\begin{equation}\label{strL}
\begin{array}{lll}
\f E_0=0,\qquad & \f E_i=E_{n+i},\qquad & \f E_{n+i}=E_i, \\[0pt]
\xi=E_0, \qquad  & \eta(E_0)=1, \qquad & \eta(E_i)=\eta(E_{n+i})=0.
\end{array}
\end{equation}

Moreover, the Riemannian metric $g$ is such that:
\begin{equation}\label{gL}
\begin{array}{l}
  g(E_0,E_0)=g(E_i,E_i)=g(E_{n+i},E_{n+i})=1, \\
  g(E_0,E_j)=g(E_j,E_k)=0,
\end{array}
\end{equation}
where $i\in\{1,\dots,n\}$ and $j, k \in\{1,\dots,2n\}$, $j\neq k$.

Let us remark that the same Lie group is considered with an
almost contact structure and a compatible Riemannian
metric in \cite{Ol} and the generated almost cosymplectic manifold is studied.
On the other hand,
the same Lie group is equipped with an almost contact structure and
B-metric in \cite{HM} and the obtained manifold is characterized. Moreover, for the latter manifold,
the case of the lowest dimension is considered and properties of the constructed
manifold are determined in \cite{HManMek}.

Let us consider the constructed almost paracontact almost paracomplex Riemannian
manifold
$(\LLL,\f, \xi, \allowbreak{}\eta, g)$ of dimension 3, i.e. for $n=1$. We use the following results from \cite{ManTav57}.
The basic components of the Levi-Civita connection $\n$ of $g$ are given by
\begin{equation}\label{nEi}
\begin{array}{ll}
    \n_{E_1}E_0=a_1E_1+a_2E_2,\qquad & \n_{E_2}E_0=a_2E_1-a_1E_2,\\[0pt]
    \n_{E_1}E_1=-\n_{E_2}E_2=-a_1E_0,\qquad & \n_{E_1}E_2=\n_{E_2}E_1=-a_2E_0,
\end{array}
\end{equation}
and the others $\n_{E_i}E_j$ are zero.
The  components $F_{ijk}=F(E_i,E_j,E_k)$
of the fundamental tensor are the following
\[
F_{101}=F_{110}=F_{202}=F_{220}=-a_2,\qquad
F_{102}=F_{120}=-F_{201}=-F_{210}=-a_1,
\]
and the other components of $F$  are zero.
Thus, it is proved the following

\begin{prop}\label{prop-exa1}
The constructed 3-dimen\-sional almost paracontact almost paracomplex Riemannian
manifold $(\LLL,\f,\xi,\eta,g)$ belongs to:
\begin{enumerate}
  \item
 $\F_4\oplus\F_9$ if and only if $a_1\neq 0$, $a_2\neq 0$;
  \item
 $\F_4$ if and only if $a_1= 0$, $a_2\neq 0$;
  \item
  $\F_9$ if and only if $a_1\neq 0$, $a_2= 0$;
    \item
  $\F_0$ if and only if $a_1= 0$, $a_2= 0$.
\end{enumerate}
\end{prop}

Let us remark that $(\LLL,\f,\xi,\eta,g)$ is a para-Sasakian paracomplex Riemannian manifold
if and only if $a_1= 0$, $a_2=1$ \cite{ManTav57}.

Using \eqref{SvK=n} and \eqref{tSvK=n}, we obtain that all components $\n^\parallel_{E_i}E_j$ and $\tn^\parallel_{E_i}E_j$ vanish.
That means the pair of of associated Schouten-van Kampen connections $\n^\parallel$ and $\tn^\parallel$ coincide with the so-called
Weitzenbock connection, the connection of the parallelization. Therefore $\n^\parallel$ and $\tn^\parallel$ have vanishing curvature, but in general nonvanishing
torsion, because it equals to $-[\,\cdot\,,\,\cdot\,]$.

A quick review of the statements proved in previous sections shows that the example in this section confirms them.


\end{document}